\documentclass[oneside,12pt]{amsart}
\usepackage{amssymb, mathrsfs, amscd}
\usepackage[all]{xy}
\usepackage{dsfont}
\usepackage{graphics}
\usepackage{graphicx}

\newtheorem{theorem}{Theorem}[section]
\newtheorem{definition}[theorem]{Definition}
\newtheorem{lemma}[theorem]{Lemma}

\newtheorem{proposition}[theorem]{Proposition}
\newtheorem{corollary}[theorem]{Corollary}

\newtheorem{remark}[theorem]{Remark}

\author{Dieter Degrijse}
\address{Department of Mathematics, Catholic University of Leuven, Kortrijk, Belgium}%
\email{Dieter.Degrijse@kuleuven-kortrijk.be}%

\author{Nansen Petrosyan}
\address{Department of Mathematics, Catholic University of Leuven, Kortrijk, Belgium}%
\email{Nansen.Petrosyan@kuleuven-kortrijk.be}%
\title[Char. classes for split Hopf alg. ext.]{Characteristic classes for cohomology of split Hopf algebra extensions}

\thanks{The first author was supported by the Research Fund K.U.Leuven.\\
\indent The second author was supported by the FWO-Flanders Research Fellowship.}%
\subjclass[2000]{Primary 18G60; Secondary 18G40, 17B56, 20J06.}%
\bigskip
\begin{document}
\maketitle

\vspace{-1mm}

\begin{abstract}
We introduce characteristic classes for the spectral sequence associated to a split short exact sequence of Hopf algebras. These  classes can be seen as obstructions for the vanishing of differentials in the spectral sequence. We give a decomposition theorem and interpret our results in the settings of group and Lie algebra extensions. As applications, we derive several results concerning the collapse of the (Lyndon-)Hochschild-Serre spectral sequence and the order of characteristic classes.
\end{abstract}
\section{Introduction}
Suppose $H$ is a group and $L$ is an $n$-dimensional integral lattice.
In \cite{CharlapVasquez}, L. Charlap and A. Vasquez defined characteristic classes for the second page of the Lyndon-Hochschild-Serre spectral sequence associated to the extension
\begin{equation}\label{eq: intro1}0 \rightarrow L \rightarrow L \rtimes H \rightarrow  H \rightarrow 0
\end{equation}
and showed that these classes can be seen as obstructions for the vanishing of differentials on the second page of the spectral sequence. In \cite{Sah}, C.-H. Sah generalized their results by defining characteristic classes with the same properties on every page of the spectral sequence. In this paper, we aim to expand this theory to split short exact sequences of Hopf algebras. The reason we use Hopf algebras and not just algebras in general is because the existence of a cup product on cohomology is vital to the theory. We now give a brief outline of the paper.

In Section \ref{sec: pre}, we introduce some notation and state a few preliminary definitions and results concerning semi-direct products of Hopf algebras. This section is largely based on R. Molnar's paper \cite{Molnar}. We merely include  it here for the reader's convenience. In Section \ref{sec: cohom}, we recall some basic facts concerning the cohomology of Hopf algebras and introduce the spectral sequence we will be studying. Then, in Section \ref{sec: char classes}, we define characteristic classes for the spectral sequence associated to a split short exact sequence of Hopf algebras.
\begin{definition} \rm
Let $t\geq 0, r\geq 2$. Suppose we have a split short exact sequence of Hopf algebras
$A \rightarrow B \rightarrow C,$ and let $(E_{\ast}(M),d_{\ast})$ be the spectral sequence associated to this extension. We define $\mathscr{M}^t_r(A,B)$ to be the subclass in $C\mbox{-mod}=\{ M \in B\mbox{-mod} \ | \ M^A=M \}$ consisting of the $B$-modules $M$ for which the differentials $d_p^{s,t}$, with source $E^{s,t}_{p}(M)$, are zero for all $s$ and all $2\leq p \leq r-1$. We say (\ref{eq: intro1}) is \emph{$(t,r)$-trivial} if $\mathrm{H}_{t}(A,k)$ is in $\mathscr{M}^t_r(A,B)$.

Assuming the split short exact sequence is $(t,r)$-trivial,
a \emph{characteristic class $v_r^t$} of the spectral sequence $(E_{\ast}(\mathrm{H}_t(A,k)),d_{\ast})$ is defined as $d^{0,t}_r([id^t])$
where $[id^t]$ is the class in $E^{0,t}_r(\mathrm{H}_t(A,k))\cong \mathrm{Hom}_C(\mathrm{H}_t(A,k),\mathrm{H}_t(A,k))$ corresponding to the identity map under the isomorphism.
\end{definition}
Also, in Section \ref{sec: char classes}, we give a generalization of a theorem of Sah (see Theorem 3 of \cite{Sah}) which shows that characteristic classes can be seen as obstructions for the vanishing of differentials in the spectral sequence.
\begin{theorem} \label{th: intro} Let $t\geq 0, r\geq 2$ and suppose we have a $(t,r)$-trivial split short exact sequence of Hopf algebras
$A \rightarrow B \rightarrow C$.
Then the following holds.
\begin{itemize}
\item[(a)] For all $s\geq0$ and for all $M \in C\mbox{-mod} $, there is a canonical surjective homomorphism
\[ \theta: E_r^{s,0}(\mathrm{H}^t(A,M)) \rightarrow E_r^{s,t}(M).\]
\item[(b)] The characteristic class $v_r^t \in E_r^{r,t-r+1}(\mathrm{H}_t(A,k))$  has the property
\[ d_r^{s,t}(x)= (-1)^s y \cdot v_r^t \] $\forall s \geq 0,  \forall x \in E_r^{s,t}(M),\forall M \in C\mbox{-mod} \ \mbox{and} \ \forall y \in  E_r^{s,0}(\mathrm{H}^t(A,M)) \ \mbox{with} \ \theta(y)=x.$
\item[(c)] $v_r^t$ is completely determined by the previous property.
\item[]
\item[(d)] Suppose we have a Hopf algebra $D$ together with a Hopf algebra map $\rho: D \rightarrow C$ that turns $A$ into a $D$-module bialgebra.  Assume that $A \rtimes D$, given the standard coalgebra structure, is a Hopf algebra and that the split short exact sequence associated to $A \rtimes D$ is also $(t,r)$-trivial (denote its characteristic classes by $w_r^t)$. Then $v_r^t$ maps to $w_r^t$ under the map induced by $\rho$ on the spectral sequences. In particular,  $v_2^t$ maps to $w_2^t$.
\end{itemize}
\end{theorem}
In Section \ref{sec: decomp}, we consider a split short exact sequence of Hopf algebras whose kernel $A$ decomposes into a tensor product of Hopf algebras $A_1 \otimes A_2$ such that the action of the quotient factors over this tensor product. For this type of extensions we prove the following decomposition theorem, which is a generalization of Petrosyan's work in \cite{Petrosyan}.
\begin{theorem}\label{intro: decomp} Let $t\geq 0$ and $r\geq 2$.
Suppose the characteristic classes ${}^{1}v_p^i$ and ${}^{2}v_p^j$, of ${}^{1}E_{\ast}(\mathrm{H}_i(A_1,k))$ and  ${}^{2}E_{\ast}(\mathrm{H}_j(A_2,k))$ respectively, are zero for all $i,j\leq t$ and $2\leq p\leq r-1$. Then the split short exact sequence
$A \rightarrow A \rtimes C \rightarrow C$
is $(t,r)$-trivial. Furthermore, we have a decomposition formula
\[ v_r^t=\sum_{i+j=m} \Big( P_{i,j}({}^{1}v_r^i\otimes [{}^{2}id^j])+ (-1)^i P_{i,j}([{}^{1}id^i]\otimes {}^{2}v_r^j)\Big). \]
\end{theorem}
\noindent Here, $P_{i,j}: {}^{1}E_r(\mathrm{H}_i(A_1,k)) \otimes_k {}^{2}E_r(\mathrm{H}_j(A_2,k)) \rightarrow E_r(\mathrm{H}_{i+j}(A,k))$ is a morphism induced by a spectral sequence pairing and $[{}^{l}id^s]$ is the element in ${}^{l}E^{0,s}_r(\mathrm{H}_s(A_l,k))$ corresponding to the identity map in $\mathrm{Hom}_C(\mathrm{H}_s(A_l,k),\mathrm{H}_s(A_l,k))$.  The following is a useful corollary.
\begin{corollary}\label{cor:intro} Suppose the spectral sequences associated to $A_1 \rtimes  C$ and  $A_2 \rtimes  C$ collapse at the second page, in coefficients $\mathrm{H}_t(A_1,k)$ and $\mathrm{H}_t(A_2,k)$ respectively, for each $t\geq 0$. Then the spectral sequence associated to $A \rtimes C$ will collapse at the second page, for all coefficients $M$ for which $M^A=M$.
\end{corollary}
In Sections \ref{sec: lie algebra} and \ref{sec: group}, we specialize to the cases of split Lie algebra and group extensions. Since the universal enveloping algebra functor and the group ring functor  respectively map split Lie algebra extensions and split group extensions to split short exact sequences of Hopf algebras, our general theory of characteristic classes applies. In these special cases, the spectral sequence under consideration is the well-known (Lyndon-)Hochschild-Serre spectral sequence. Using Theorem \ref{th: intro}  in the context of split Lie algebra extensions, we obtain the following results concerning their collapse.
\begin{corollary} Let $k$ be a field of characteristic zero and suppose $ 0 \rightarrow \mathfrak{n} \rightarrow \mathfrak{n} \rtimes_{\varphi} \mathfrak{h} \rightarrow \mathfrak{h}$ is a split extension of finite dimensional Lie algebras. Assume that $\mathrm{Der}(\mathfrak{n})$ has a semi-simple Lie subalgebra $\mathfrak{s}$ such that $\varphi$ factors through $\mathfrak{s}$, i.e.
$\varphi: \mathfrak{h} \rightarrow \mathfrak{s} \subseteq \mathrm{Der}(\mathfrak{n})$.
Then, the Hochschild-Serre spectral sequence $(E_{\ast}(M),d_{\ast})$ associated to the split extension will collapse at the second page for any $\mathfrak{h}$-module $M$.
\end{corollary}
\begin{corollary}
 If $0 \rightarrow \mathfrak{n} \rightarrow \mathfrak{n} \rtimes_{\varphi} \mathfrak{h} \rightarrow \mathfrak{h} \rightarrow 0$ is a split extension of Lie algebras such that $\varphi(\mathfrak{h})$ is one-dimensional, then the Hochschild-Serre spectral sequence associated to this extension will collapse at page two for all coefficients in $\mathfrak{h}$-mod.
\end{corollary}
A theorem of Barnes (see \cite[Th. 3]{Barnes1}) asserts that the Hochschild-Serre spectral sequence collapses at page two for all $\mathfrak{h}$-module coefficients  when the kernel $\mathfrak{n}$ is finite dimensional and abelian. If the base field has characteristic zero, using Corollary \ref{cor:intro}, we can expand this result to extensions with reductive kernels.
\begin{theorem}
Suppose $k$ is a field of characteristic zero. Consider the split extension of finite dimensional Lie algebras
$0 \rightarrow \mathfrak{n} \rightarrow \mathfrak{g} \rightarrow \mathfrak{h} \rightarrow 0,$
and suppose that $\mathfrak{n}$ is a reductive Lie algebra. Then the Hochschild-Serre spectral sequence associated to this extension will collapse at the second page for all $M \in \mathfrak{h}\mbox{-mod}$.
\end{theorem}
It turns out that in the group case, for the split group extensions of type (\ref{eq: intro1}), all characteristic classes have finite order. In Section \ref{sec: group}, we expand several results of Sah from \cite{Sah} about the order of characteristic classes. Because of the technical formulations we omit the statements here and give a corollary instead.
\begin{corollary}The Lyndon-Hochschild-Serre spectral sequence associated to (\ref{eq: intro1}) collapses at the second page for all $H$-modules $M$
if and only if the characteristic classes $v_r^{p^n}$ are zero for all primes $p$ with $(p-1)| (r-1)$, for all $n \in \mathbb{N}_0$ and all $r\geq 2$.
\end{corollary}

We would like to thank the referee for the thorough, constructive  and quite valuable comments and suggestions. They led to many improvements and corrections of the manuscript for which we are indebted.

\section{Preliminaries on semi-direct products of Hopf algebras} \label{sec: pre}
We assume the reader is familiar with the basic concepts from the theory of Hopf algebras. A good introduction can, for example, be found in \cite{Cartier}, \cite{Milnor} and \cite{SweedlerBook}.
Throughout the paper, we also assume a working knowledge of homological algebra, spectral sequences (see \cite{CartanEilenberg}, \cite{McClearly}, and \cite{Weibel}), and some introductory notions from the cohomology theories of groups and Lie algebras (see \cite{Brown}, \cite{Knapp} and \cite{Weibel}).

Let us now establish some notation. Suppose $A$ is a Hopf algebra over the field $k$. With $A \otimes A$ we mean the tensor product over the ground field $k$. Moreover, when there is no subscript present, the tensor product will always be over the ground field $k$. Multiplication in $A$ will be denoted by
\[ m_A : A \otimes A \rightarrow A : a\otimes b \mapsto ab. \]
For the comultiplication, we will use the standard Sweedler notation
\[ \Delta_A: A \rightarrow A \otimes A : a \mapsto \sum a_{(1)}\otimes a_{(2)}. \]
So by coassociativity, we may unambiguously write
\[ (\Delta_A \otimes \mathrm{Id})\circ \Delta_A (b)=(\mathrm{Id} \otimes \Delta_A)\circ \Delta_A (b)=\sum  b_{(1)}\otimes b_{(2)}\otimes b_{(3)}.\]
The unit and counit of $A$ are respectively denoted by
$\eta_A : k \rightarrow A$ and $\varepsilon_{A}: A \rightarrow k$. $A$ also has an antipode $S_A: A \rightarrow A$. The augmentation ideal $A^{+}$ of $A$ is by definition the kernel of $\varepsilon_A$.

Suppose $A,B$ and $H$ are Hopf algebras (in fact, we only need them to be (bi)algebras here) such that $A$ and $B$ are $H$-modules, that is, we have an $H$-module structure map
\[ \tau_A: H \otimes A \rightarrow A : h \otimes a \mapsto h \cdot a\]
and similarly a structure map $\tau_B$ for $B$.
By $\mathrm{Alg}(A,B)$, we mean the set of algebra maps from $A$ to $B$, and by $\mathrm{Alg}_H(A,B)$, the set of algebra maps from $A$ to $B$ that are also $H$-module maps.
We say $A$ is an $H$-module algebra if $m_A$ and $\eta_A$ are $H$-module maps. Dually,
we say $A$ is an $H$-module coalgebra if $\Delta_A$ and $\varepsilon_A$ are $H$-module maps (here $k$ becomes an $H$-module via $\varepsilon_H$ and $A\otimes A$ becomes an $H$-module via $\Delta_H$).
Combining these two, we say $A$ is an $H$-module bialgebra if it is both an $H$-module algebra and an  $H$-module coalgebra. Starting with bialgebras that are $H$-comodules, one can define the notion of $H$-comodule (co/bi)algebra in a similar way. \\
\indent Next, we recall the notion of (split) short exact sequences of Hopf algebras.
Let $\mathscr{H}$ be the category of Hopf algebras over the field $k$. It is well-known that kernels and cokernels exist in this category. If $u: H \rightarrow L$ is a map of Hopf algebras, following Molnar's notation in \cite{Molnar}, we will denote the categorical kernel and cokernel of $u$ in $\mathscr{H}$ by KER(u) and COK(u), respectively. One can check that the kernel and cokernel of $u$ are given by
\begin{itemize}
\item[-] $\mathrm{KER}(u)= (\mathscr{K}(\mathrm{Ker}(u)),j)$
\item[-] $\mathrm{COK}(u)= (L/ \mathscr{J}(\mathrm{Im}(u)),p)$,
\end{itemize}
where $j$ is the canonical inclusion, $p$ the canonical projection and
\begin{itemize}
\item[-] $\mathscr{K}(\mathrm{Ker}(u)) = \mbox{the largest Hopf subalgebra of $H$ contained in $\mathrm{Ker}(u)+k$}$
\item[-] $\mathscr{J}(\mathrm{Im}(u)) = \mbox{the smallest Hopf ideal of $L$ containing $\mathrm{Im}(u)^{+}$.}$
\end{itemize}
\begin{definition} \rm Let $H$ be a Hopf algebra. The \emph{left adjoint action} of $H$ is defined by
\[ \mathrm{ad}^l_H : H \otimes H \rightarrow H : g \otimes h \mapsto \sum g_{(1)}h S_{H}(g_{(2)}), \]
and \emph{right adjoint action} of $H$ is defined as
\[ \mathrm{ad}^r_H : H \otimes H \rightarrow H : h \otimes g \mapsto \sum S_H(g_{(1)})h g_{(2)}. \]
Dually, the \emph{left adjoint coaction} of $H$ is defined by
\[ \mathrm{co}^l_H : H \rightarrow H \otimes H  : g \mapsto \sum g_{(1)}S_H(g_{(3)})\otimes g_{(2)}, \]
and the \emph{right adjoint coaction} of $H$ is defined as
\[ \mathrm{co}^r_H : H \rightarrow H \otimes H  : g \mapsto \sum g_{(2)}\otimes S_H(g_{(1)})g_{(3)}. \]
\end{definition}
In \cite{Molnar}, it is shown that $\mathrm{ad}^l_H$ turns $H$ into an $H$-module algebra and $\mathrm{co}^l_H$ turns $H$ into a $H$-comodule coalgebra.
\begin{definition} \rm Let $u: H \rightarrow L$ be a map of Hopf algebras. We say $u$ is \emph{normal} if $u(H)$ is stable under the left and right adjoint actions of $L$. We say $u$ is \emph{conormal} if $ker(u)$ is a left and right  $H$-cosubmodule of $H$ under the left and right adjoint coactions of $H$, respectively.
\end{definition}
The following lemma is proven in \cite{Molnar}.
\begin{lemma}\label{lemma: normal} If $u: H \rightarrow L$ is a normal map between Hopf algebras $H$ and $L$, then
\[ u(H^{+})L=Lu(H^{+}). \]
\end{lemma}
We are now ready to explain what we mean by``a short exact sequence of Hopf algebras''.
\begin{definition}\rm Let $A \xrightarrow{i} B \xrightarrow{\pi} C,$
be a sequence of Hopf algebra maps. We say this sequence is \emph{exact} if $i$ is normal, $\pi$ is conormal, $(A,i)=\mathrm{KER}(\pi)$ and $(C,\pi)=\mathrm{COK}(i)$. We say an exact sequence is \emph{split} if there is a Hopf algebra map $\sigma: C \rightarrow B$, called a section, such that $\pi \circ \sigma = \mathrm{Id_C}$.
\end{definition}
Now, let us define the semi-direct product of two Hopf algebras.
\begin{definition} \label{def: smash}\rm Let $A$ and $C$ be two Hopf algebras such that $A$ is a $C$-module bialgebra. The semi-direct product (or smash product) $A \rtimes C$ of $A$ and $C$ is an algebra with the underlying vector space structure of $A \otimes C$ and the multiplication in $A \rtimes C$ given by
\[ (a \otimes g)(b \otimes h)= \sum a(g_{(1)}\cdot b)\otimes g_{(2)}h. \]
\end{definition}
There are algebra injections
$i: A \rightarrow A \rtimes C$ and $j: C \rightarrow A \rtimes C$,
defined in the obvious way, that give the semi-direct product the following universal property: if $B$ is any algebra then for any $g \in \mathrm{Alg}(C,B)$ and any $f \in \mathrm{Alg}_C(A,B)$, the map
\[ f \rtimes g: A \rtimes C \rightarrow B :   a \otimes c \mapsto f(a)g(c) \]
is the unique algebra map such that $(f \rtimes g )\circ i =f$ and $(f \rtimes g )\circ j =g$. Here, $B$ becomes a $C$-module via the adjoint action of $C$ on $B$ associated to $g$, i.e. via the following structure map
\[  C \otimes B \rightarrow B : c \otimes b \mapsto \sum g(c_{(1)})b g(S_{C}(c_{(2)})). \]
\indent Note that given two coalgebras $A$ and $C$, we can always give $A\otimes C$ the structure of a coalgebra by defining the coproduct as
$\Delta(a \otimes g)= \sum a_{(1)}\otimes g_{(1)} \otimes a_{(2)}\otimes g_{(2)}.$
While the counit is given by
$\varepsilon: A \otimes C \rightarrow k : a \otimes g \mapsto \varepsilon_A(a)\varepsilon_C(g).$
One can easily check that these operations turn $A\otimes C$ into a coalgebra. We call this the standard coalgebra structure on $A \otimes B$.

It is also important to note that the algebra structure on $A \rtimes C$ need not be compatible with the standard coalgebra structure we have just defined. So, in general, $A \rtimes C$ is not a Hopf algebra.  However, in \cite{Molnar}, R. Molnar shows that if $C$ is a cocommutative Hopf algebra then $A \rtimes C$ is a Hopf algebra with the standard coalgebra structure. Also, in \cite{Radford}, D. Radford gives explicit, necessary and sufficient conditions for $A \rtimes C$ to be a Hopf algebra with the standard coalgebra structure. Actually, Radford considers the more general case where $A$ is also a $C$-comodule coalgebra and the coalgebra structure on $A \otimes C$ is induced by this comodule structure. But our situation can be seen as a special case of this.

The following theorem shows that split short exact sequences of Hopf algebras entail semi-direct products of Hopf algebras that are a fortiori Hopf algebras.
\begin{theorem}{\normalfont(Molnar, \cite{Molnar})} Suppose $A \xrightarrow{i} B \xrightarrow{\pi} C$
is a split short exact sequence of Hopf algebras with a section $\sigma: C \rightarrow B$. Then $A$ has a structure of a $C$-module bialgebra such that $i \in \mathrm{Alg}_C(A,B)$ and for which the universal property of the semi-direct product entails a Hopf algebra isomorphism
$i \rtimes \sigma:   A \rtimes C \xrightarrow{\cong} B,$ where $A \rtimes C$ is given the standard coalgebra structure.
%Here, the coalgebra structure on $A \rtimes C$ is the one discussed above.
\end{theorem}
\begin{remark}\label{remark: some remarks} \rm
\begin{itemize}\item[]
\item[-] If $A \rightarrow B \rightarrow C$ is a split short exact sequence of Hopf algebras then $B$ is a free $A$-module via the inclusion of $A$ into $B$. Indeed, by the theorem we may assume that $B=A \rtimes C$. So we need to show that $A \rtimes C$ is a free $A$-module via the inclusion $A \rightarrow A \rtimes C : a \mapsto a \otimes 1$. But this is clear, since if $\{c_{\alpha}\}_{\alpha \in I}$ is a vector space basis for $C$, then $\{1 \otimes c_{\alpha}\}_{\alpha \in I}$ is an $A$-module basis for $A \rtimes C$.
\item[-] The $C$-module structure on $A$ is obtained as follows: by normality, $A$ can be seen as a $B$-module via the adjoint action of $B$, then $\sigma: C \rightarrow B$ gives $A$ a $C$-module structure.

    \smallskip

\item[-] Let $A$ and $C$ be two Hopf algebras such that $A$ is a $C$-module bialgebra. If $A \rtimes C$ equipped with the standard coalgebra structure is a Hopf algebra, then
$A \xrightarrow{i} A \rtimes C \xrightarrow{\pi} C$
is a split short exact sequence of Hopf algebras, with a section equal to the canonical inclusion $j: C \rightarrow A \rtimes C$ and $\pi$ defined by $\pi: A \rtimes C \rightarrow C : a \otimes c \mapsto \varepsilon_A(a)c.$

\smallskip

\item[-] If in Definition \ref{def: smash}, $A$ is a trivial $C$-module, meaning $c\cdot a=\varepsilon_C(c)a$ for all $c \in C$ and $a \in A$, then the product on $A\rtimes C$ is just the ordinary tensor algebra structure (then we just write $A \otimes C $ instead of $A \rtimes C$). In this case, the algebra structure is compatible with the standard coalgebra structure and thus, $A \otimes C$ is a Hopf algebra.

\end{itemize}
\end{remark}

\section{(Co)homology of Hopf algebras}\label{sec: cohom}
Let $A$ be a Hopf algebra over the field $k$. Given $A$-modules $M$ and $N$, we turn $M \otimes_k N$ into an $A$-module by defining
$a(m \otimes n)= \sum a_{(1)}m \otimes a_{(2)}n.$
Also, we turn $\mathrm{Hom}_k(M,N)$ into an $A$-module via
$(af)(m)=\sum a_{(1)}f(S(a_{(2)})m).$
Furthermore, these $A$-module structures are compatible with each other in the following sense. For $A$-modules $M$, $N$ and $K$, there is a natural isomorphism
\begin{equation} \label{eq: comp hom tensor mod}\Psi: \mathrm{Hom}_A (M \otimes_k N, K) \xrightarrow{\cong} \mathrm{Hom}_A(M,\mathrm{Hom}_k(N,K)) \end{equation}
where $\Psi(f)(m)(n)=f(m \otimes n)$.
\begin{definition}\rm
Let $M$ be an $A$-module. The \emph{invariants} of $M$ is defined as the submodule
$M^{A}= \{ m \in M \ | \ A^{+}m=0 \}$
and the \emph{coinvariants} of $M$, as the quotient $M_{A}= M/A^{+}M.$
\end{definition}
One can easily check that $-^{A}$ and $-_{A}$ are functors from $A$-mod to $k$-mod. Furthermore, we have the natural isomorphisms
$\mathrm{Hom}_A(k,-) \cong -^{A}$
and
$k\otimes_A - \cong -_{A}.$ Note that the isomorphism (\ref{eq: comp hom tensor mod}) implies that
$\mathrm{Hom}_A(N,K) \xrightarrow{\cong} \mathrm{Hom}_k(N,K)^A$
for all $A$-modules $N$ and $K$.

%We can now make the following definition.
\begin{definition}\rm
Let $M$ be an $A$-module. The $n$-th homology of $A$ with coefficients in $M$ is defined as
\[ \mathrm{H}_n(A,M):= \mathrm{Tor}^A_n(k,M)=\mathrm{L}_n(-_{A})(M). \]
The $n$-th cohomology of $A$ with coefficients in $M$ is defined as
\[ \mathrm{H}^n(A,M):= \mathrm{Ext}_A^n(k,M)=\mathrm{R}^n(-^{A})(M). \]
\end{definition}
\begin{remark}\rm It is well-known that this definition of (co)homology is isomorphic to the Hochschild (co)homology of $A$ with coefficients in $M$, where $M$ is turned into a bimodule via the augmentation map $\varepsilon$. (see \cite{CartanEilenberg}).
\end{remark}
Now, suppose we have a split short exact sequence of Hopf algebras $A \xrightarrow{i} B \xrightarrow{\pi} C.$
When $M$ is a $B$-module, one can verify using Lemma \ref{lemma: normal} that $M^{A}$ and $M_{A}$ are also $B$-modules. Because $A$ acts trivially on $M^A$ and $M_{A}$, it follows that we can give these spaces a $C$-module structure.
\begin{lemma}\label{prop: pairing}
If $M$ is a $B$-module such that $M^A=M$, then we have an isomorphism of $C$-modules
\[ \mathrm{H}^n(A,M) \xrightarrow{\cong} \mathrm{Hom}_k(\mathrm{H}_n(A,k),M)\]
\end{lemma}
\begin{proof}
This follows directly from the Universal Coefficient Theorem.
\end{proof}
Let $\mathscr{T}$ be the exact functor that turns  $C$-modules into  $B$-modules via the map $\pi$.
\begin{lemma}
If $N$ is a $B$-module and $M$ is a $C$-module, then we have a natural isomorphism
\[ \mathrm{Hom}_B(\mathscr{T}(M),N) \xrightarrow{\cong} \mathrm{Hom}_C(M,N^A), \]
which implies that the functor $\mathscr{T}$ is left adjoint to the functor $-^A$, and that
$-^B \cong -^C \circ -^A.$
\end{lemma}
Since the functor $-^A: B\mbox{-mod} \rightarrow C\mbox{-mod}$ is right adjoint to an exact functor, it follows that $-^A$ preserves injective modules. Furthermore, because $B$ is a free $A$-module (see Remark \ref{remark: some remarks}), $B$-resolutions can be used to compute the right derived functors of $-^A: A\mbox{-mod} \rightarrow k\mbox{-mod}$. This implies that we have natural isomorphisms of $C$-modules $\mathrm{R}^{\ast}(-^{A})(M)\cong \mathrm{H}^{\ast}(A,M)$ for every $B$-module $M$. These facts, together with the composition
$-^B \cong -^C \circ -^A$, give us a convergent first quadrant cohomological Grothendieck spectral sequence for every $B$-module $M$
\begin{equation}\label{eq: spec seq} E_2^{p,q}(M)= \mathrm{H}^p(C,\mathrm{H}^q(A,M)) \Rightarrow \mathrm{H}^{p+q}(B,M). \end{equation}

\begin{lemma}
If $A \xrightarrow{i} B \xrightarrow{\pi} C$ is a split short exact sequence of Hopf algebras and $M$ is a $B$-module such that $M^A=M$, then the differentials $d_r^{\ast,r-1}$ from $E^{\ast,r-1}_r(M)$, are zero for all $r\geq2$.
\end{lemma}

\begin{proof}
The section is a Hopf algebra map $\sigma: C \rightarrow B$, such that $\pi \circ \sigma=\mathrm{Id}$. Factoriality then entails that the induced maps
$ \mathrm{H}^n(C,M) \rightarrow \mathrm{H}^n(B,M)$ are injective for all $n$. Since these maps are given by the composition
\[  \mathrm{H}^n(C,M) = E_2^{n,0}(M) \twoheadrightarrow E_3^{n,0}(M) \twoheadrightarrow \ldots \twoheadrightarrow E_{\infty}^{n,0}(M) \hookrightarrow \mathrm{H}^n(B,M), \]
we conclude that $d_r^{\ast,r-1}=0$ for all $r\geq2$.
\end{proof}
Since $A,B$ and $C$ are Hopf algebras, their cohomology is endowed with a cup product.
Now, let us suppose that there is a pairing of $B$-modules $M \otimes_k N \rightarrow K.$
Then this, together with the cup product, induces a pairing of spectral sequences
\[ E_r^{p,q}(M)\otimes_k E_r^{l,m}(N) \rightarrow E_r^{p+l,q+m}(K): a\otimes b \mapsto a\cdot b. \]
Moreover, this pairing will satisfy
\[ d^{p+l,q+m}_r(ab)= d^{p,q}_r(a)\cdot b+ (-1)^{p+q}a \cdot d^{l,m}_r(b). \]
%new section
\section{Characteristic Classes} \label{sec: char classes}
Suppose $A \xrightarrow{i} B \xrightarrow{\pi} C$ is a split short exact sequence of Hopf algebras and consider the spectral sequence (\ref{eq: spec seq}).
\begin{definition} \rm
Let $t\geq 0, r\geq 2$. We define $\mathscr{M}^t_r(A,B)$ to be the class of $B$-modules $M$ such that $M^A=M$ and for which the differentials $d_p^{s,t}$ with source $E^{s,t}_{p}(M)$ are zero for all $s$ and all $2\leq p \leq r-1$. We say that the split extension is \emph{$(t,r)$-trivial} if $\mathrm{H}_{t}(A,k) \in \mathscr{M}^t_r(A,B)$.
\end{definition}
Now, suppose $M$ is a $B$-module such that $M^A=M$ and assume that the extension is $(t,r)$-trivial.
Then, by lemma \ref{prop: pairing}, we have a non-degenerate $C$-pairing
\[ \mathrm{H}^t(A,M)\otimes_k \mathrm{H}_t(A,k) \rightarrow M. \]
As stated earlier, this induces a spectral sequence pairing
\[ E_r^{p,q}(\mathrm{H}^t(A,M))\otimes_k E_r^{l,m}(\mathrm{H}_t(A,k)) \rightarrow E_r^{p+l,q+m}(M)  \ . \]
We also have isomorphisms
\begin{eqnarray*}
E^{0,t}_r(\mathrm{H}_t(A,k)) & = & E^{0,t}_2(\mathrm{H}_t(A,k)) \\
& = & \mathrm{H}^0(C,\mathrm{H}^t(A,\mathrm{H}_t(A,k))) \\
& \cong & \mathrm{H}^0(C,\mathrm{Hom}_k(\mathrm{H}_t(A,k),\mathrm{H}_t(A,k))) \\
& \cong & \mathrm{Hom}_k(\mathrm{H}_t(A,k),\mathrm{H}_t(A,k))^C \\
& \cong & \mathrm{Hom}_C(\mathrm{H}_t(A,k),\mathrm{H}_t(A,k)).
\end{eqnarray*}
\begin{definition}\rm Let $t\geq 0, r\geq 2$ and suppose that the extension is $(t,r)$-trivial.
A \emph{characteristic class} of the spectral sequence $(E_{\ast}(\mathrm{H}_t(A,k)),d_{\ast})$ is defined as
\[ v_r^t(A):= d^{0,t}_r([id^t]), \]
where $[id^t]$ is the image in $E^{0,t}_r(\mathrm{H}_t(A,k))$ of the identity map in $\mathrm{Hom}_C(\mathrm{H}_t(A,k),\mathrm{H}_t(A,k))$ under the isomorphisms above.
\end{definition}

If $M$ is a $B$-module such that $M^A=M$, it turns out that these characteristic classes can be seen as obstructions to the vanishing of differentials in $(E_{\ast}(M),d_{\ast})$.
\begin{theorem}\label{th: sah} Let $t\geq 0, r\geq 2$ and suppose we have a $(t,r)$-trivial split short exact sequence of Hopf algebras
$A \rightarrow B \rightarrow C.$
Then the following holds.
\begin{itemize}
\item[(a)] For all $s\geq0$ and for all $M \in \mathscr{M}^t_r(A,B)$, there is a canonical surjective homomorphism
\[ \theta: E_r^{s,0}(\mathrm{H}^t(A,M)) \rightarrow E_r^{s,t}(M).\]
\item[(b)] The characteristic class $v_r^t \in E_r^{r,t-r+1}(\mathrm{H}_t(A,k))$  has the property
\[ d_r^{s,t}(x)= (-1)^s y\cdot v_r^t \] $\forall s \geq 0,  \forall x \in E_r^{s,t}(M),\forall M \in \mathscr{M}^t_r(A,B) \ \mbox{and} \ \forall y \in  E_r^{s,0}(\mathrm{H}^t(A,M)) \ \mbox{with} \ \theta(y)=x.$
\item[(c)] $v_r^t$ is completely determined by the previous property.
\item[]
\item[(d)]  Suppose we have a Hopf algebra $D$ together with a Hopf algebra map $\rho: D \rightarrow C$ that turns $A$ into a $D$-module bialgebra. Assume that $A \rtimes D$, given the standard coalgebra structure, is a Hopf algebra and that the split short exact sequence associated to $A \rtimes D$ is also $(t,r)$-trivial (denote its characteristic classes by $w_r^t)$. Then $v_r^t$ maps to $w_r^t$ under the map induced by $\rho$ on the spectral sequences. In particular,  $v_2^t$ maps to $w_2^t$.
\end{itemize}
\end{theorem}
\begin{proof}
Since $\mathrm{H}^{t}(A,M)^A=\mathrm{H}^{t}(A,M)$, we have a canonical isomorphism
\begin{eqnarray*}
E^{s,0}_2(\mathrm{H}^t(A,M)) & \cong & \mathrm{H}^s(C,\mathrm{H}^t(A,M)) \\
& \cong & E^{s,t}_2(M).
\end{eqnarray*}
The splitting guarantees that $E^{s,0}_r(\mathrm{H}^t(A,M))=E^{s,0}_2(\mathrm{H}^t(A,M))$. In addition, if $M \in \mathscr{M}^t_r(A,B)$, then it follows that $E^{s,t}_r(M)$ is a quotient module of $E^{s,t}_2(M)$. Combining these observations, we find a canonical surjection
\[ \theta: E_r^{s,0}(\mathrm{H}^t(A,M)) \rightarrow E_r^{s,t}(M),\]
proving (a). \\
\indent To prove (b), we fix  $s\geq0$, $M \in \mathscr{M}^t_r(A,C)$ and $x \in E_r^{s,t}(M)$. We have the following commutative diagram
\[
\xymatrix{
  E_r^{s,0}(\mathrm{H}^t(A,M))   \otimes_k   E_r^{0,t}(\mathrm{H}_t(A,k)) \ar[r] &  E_r^{s,t}(M) \\
  E_2^{s,0}(\mathrm{H}^t(A,M))  \otimes_k   E_2^{0,t}(\mathrm{H}_t(A,k))\ar[u]^{\cong} \ar[r] &  E_2^{s,t}(M) \ar@{->>}[u]
}\]
Recall that we can take $[id^t] \in  E_2^{0,t}(\mathrm{H}_t(A,k))$. We have that $z \cdot [id^t]=z$ for all $z \in  E_2^{s,0}(\mathrm{H}^t(A,M))= E_2^{s,t}(M)$.
Indeed, if we fix $[id^t] \in  E_2^{0,t}(\mathrm{H}_t(A,k))$, then the map $p:E_2^{s,0}(\mathrm{H}^t(A,M))=\mathrm{H}^s(C,\mathrm{H}^t(A,M)) \rightarrow E_2^{s,t}(M)=\mathrm{H}^s(C,\mathrm{H}^t(A,M)): z \mapsto z \cdot [id^t]$ is a map induced by the $C$-module homomorphism $\mathrm{H}^t(A,M)=\mathrm{Hom}_k(\mathrm{H}_t(A,k),M) \rightarrow  \mathrm{H}^t(A,M)=\mathrm{Hom}_k(\mathrm{H}_t(A,k),M): f \mapsto f \circ id^t $ where $id^t$ is the identity morphism in $\mathrm{Hom}_k(\mathrm{H}_t(A,k),\mathrm{H}_t(A,k))$. Since this $C$-module homomorphism is clearly the identity morphism, the induced map $p$ on cohomology is also the identity map.

Now, since $\theta$ is surjective, there is an element $y \in  E_r^{s,0}(\mathrm{H}^t(A,M))$ such that $\theta(y)=x$.
The commutativity of the diagram then implies that $y \cdot [id^t] = x$. Now, by using the product formula, we find
\begin{eqnarray*}
d_r^{s,t}(x) & = & d_r^{s,t}(y \cdot [id^t]) \\
& = & d_r^{s,0}(y)\cdot[id^t] + (-1)^sy \cdot d_r^{0,t}([id^t]) \\
& = & (-1)^s y \cdot v_r^t \ ,
\end{eqnarray*}
which proves (b).
This last equality follows from the definition of $v_r^t$ and the fact that $d_r^{s,0}$ lands in $E_r^{s+r,1-r}(M)$, which is zero.

To prove (c), we consider the special case $s=0$ and $M=\mathrm{H}_t(A,k)$. Notice that $\theta$ becomes the identity map (under identifications). Now, suppose $u_r^t$ also satisfies property (b) and set $x=[id^t] \in  E_r^{0,t}(M)$. It follows that
\[ v_r^t = d_r^{0,t}(x)= \theta^{-1}(x) \cdot u_r^t=[id^t] \cdot u_r^t.  \]
Furthermore, since our multiplication amounts to composition with the identity map in $\mathrm{Hom}_C(\mathrm{H}_t(A,k),\mathrm{H}_t(A,k))$ we find $v_r^t = u_r^t$,
proving (c).

Finally, to prove (d), we let $({'E}_{\ast}, {'d}_{\ast})$ be the spectral sequence associated to $A \rtimes D$. Note that the map induced by $\rho$ on the spectral sequences maps $[id^t] \in E_2^{0,t}(\mathrm{H}_t(A,k))=\mathrm{Hom}_C(\mathrm{H}_t(A,k),\mathrm{H}_t(A,k))$ to $[id^t] \in {'E}_2^{0,t}(\mathrm{H}_t(A,k))=\mathrm{Hom}_D(\mathrm{H}_t(A,k),\mathrm{H}_t(A,k))$. Since the map induced by $\rho$ on the spectral sequences commutes with the differentials and $v_r^t= d^{0,t}_r([id^t]), w_r^t= {'d}^{0,t}_r([id^t])$, we conclude that  $v_r^t$ is mapped to $w_r^t$.
\end{proof}
Using this theorem inductively, one can easily see that $\mathscr{M}_r^t(A,B)=C\mbox{-mod}$ if and only if the extension is $(t,r)$-trivial.
\begin{corollary} Let $t\geq 0, r\geq 2$.
Then $\mathscr{M}_r^t(A,B)=C\mbox{-mod}$ if and only if the edge differentials
\[ d_m^{0,t}: E_m^{0,t}(\mathrm{H}_t(A,k)) \rightarrow E_m^{m,t-m+1}(\mathrm{H}_t(A,k)) \]
are zero for all $2 \leq m < r$. So, $\mathscr{M}_r^t(A,B)=C\mbox{-mod}$  if and only if $A \rtimes C$ is $(t,r)$-trivial.
\end{corollary}
\begin{corollary}
The spectral sequence $(E_{\ast}(M),d_{\ast})$ collapses at the second page for all coefficients $M \in C\mbox{-mod}$ if and only if the edge differentials
\[ d_m^{0,t}: E_m^{0,t}(\mathrm{H}_t(A,k)) \rightarrow E_m^{m,t-m+1}(\mathrm{H}_t(A,k)) \]
are zero for all $t\geq 0$ and $m \geq 2$.
Said differently, the spectral sequence will collapse at the second page for all coefficients $M \in C\mbox{-mod}$  if and only if it collapses with coefficients $\mathrm{H}_t(A,k)$, for all $t\geq 0$.
\end{corollary}
\section{A decomposition Theorem} \label{sec: decomp}
Let $A_1,A_2$ and $C$ be Hopf algebras such that $A_1$ and $A_2$ are $C$-module bialgebras and assume that $A_1 \rtimes C$ and $A_2 \rtimes C$ are Hopf algebras with the standard coalgebra structure. Then $A= A_1 \otimes A_2$ naturally becomes a $C$-module bialgebra and one can check that $A \rtimes C$ is also a Hopf algebra with the standard coalgebra structure. It follows that we have three split short exact sequence of Hopf algebras; namely
\begin{equation} \label{eq: decomp_i}
A_i \rightarrow A_i \rtimes C \rightarrow C
\end{equation} for $i=1,2$ and
\begin{equation} \label{eq: decomp}
A \rightarrow A \rtimes C \rightarrow C .
\end{equation}
 We denote the spectral sequence associated to (\ref{eq: decomp_i}) by $({}^{i}E_{\ast},{}^{i}d_{\ast})$, and the spectral sequence associated to (\ref{eq: decomp}) by $(E_{\ast},d_{\ast})$.
It follows from the K\"{u}nneth formula that, for all $i$ and $j$, we have a $C$-pairing
\[ \mathrm{H}_i(A_1,k)\otimes \mathrm{H}_j(A_2,k) \rightarrow \mathrm{H}_{i+j}(A,k). \]
This induces a spectral sequence pairing
\[ E^{p,q}_r(\mathrm{H}_i(A_1,k)) \otimes_k E^{l,m}_r(\mathrm{H}_j(A_2,k)) \rightarrow E^{p+l,q+m}_r(\mathrm{H}_{i+j}(A,k)).\] Note that the natural maps of $A$ onto $A_1$ and $A_2$ induce two homomorphisms of spectral sequences
\begin{align*}
{}^{1}\Phi &: {}^{1}E_r(\mathrm{H}_i(A_1,k)) \rightarrow E_r(\mathrm{H}_i(A_1,k)), \\
{}^{2}\Phi &: {}^{2}E_r(\mathrm{H}_j(A_2,k)) \rightarrow E_r(\mathrm{H}_j(A_2,k)).
\end{align*}
Together with the pairing, these entail the morphism
\[ P_{i,j}: {}^{1}E^{p,q}_r(\mathrm{H}_i(A_1,k)) \otimes_k {}^{2}E^{l,m}_r(\mathrm{H}_j(A_2,k)) \rightarrow E^{p+l,q+m}_r(\mathrm{H}_{i+j}(A,k)): x\otimes y \mapsto {}^{1}\Phi(x){}^{2}\Phi(y)\]
for all $i$ and $j$.
\begin{lemma}\label{lemma: product formula} Let $x \in {}^{1}E^{p,q}_r(\mathrm{H}_i(A_1,k))$ and $y \in {}^{2}E^{l,m}_r(\mathrm{H}_j(A_2,k))$. Then
\[ d_r^{p+l,q+m}(P_{i,j}(x\otimes y)) = P_{i,j}({}^{1}d_r^{p,q}(x)\otimes y)+ (-1)^{p+q} P_{i,j}(x\otimes {}^{2}d_r^{l,m}(y)). \]
\end{lemma}
\begin{proof}
This follows immediately from the product formula for pairings and the fact that ${}^{1}\Phi$ and ${}^{2}\Phi$ commute with the differentials.
\end{proof}
We can now derive the following.
\begin{theorem}\label{th: decomp} Let $t\geq 0$ and $r\geq 2$.
Suppose the respective characteristic classes ${}^{1}v_p^i$ and ${}^{2}v_p^j$, of ${}^{1}E_{\ast}(\mathrm{H}_i(A_1,k))$ and  ${}^{2}E_{\ast}(\mathrm{H}_j(A_2,k))$, are zero for all $i,j\leq t$ and $2\leq p\leq r-1$. Then the split short exact sequence
$A \rightarrow A \rtimes C \rightarrow C$
is $(t,r)$-trivial. Moreover, we have a decomposition formula
\[ v_r^t=\sum_{i+j=m} \Big( P_{i,j}({}^{1}v_r^i\otimes [{}^{2}id^j])+ (-1)^i P_{i,j}([{}^{1}id^i]\otimes {}^{2}v_r^j)\Big). \]
\end{theorem}
\begin{proof} Proceeding by induction, let $2 \leq p \leq r-1$ and suppose that
\begin{equation}\label{eq: decompext}  A \rightarrow A \rtimes C \rightarrow C  \end{equation}
is $(t,p)$-trivial. So, $v_p^t$ is defined and equal to $d_p^{0,t}([id^t])$, where $id^t$ is
the identity map in $\mathrm{Hom}_{C}(\mathrm{H}_t(A,k),\mathrm{H}_t(A,k))$. Similarly, we have ${}^{1}v_p^i={}^{1}d_p^{0,i}([{}^{1}id^i])$ and ${}^{2}v_p^j={}^{2}d_p^{0,j}([{}^{2}id^j])$.
Since
\[ \mathrm{H}_t(A,k)= \bigoplus_{i+j=t}\mathrm{H}_i(A_1,k)\otimes_k \mathrm{H}_j(A_2,k), \]
we have a decomposition $id^t=\sum_{i+j=t}id_{ij}$, with $id_{ij} \in \mathrm{Hom}_{C}(\mathrm{H}_t(A,k),\mathrm{H}_t(A,k))$ defined by
\[ id_{ij}(x\otimes y) = \big\{ \begin{array}{cc}
                                  x \otimes y & \mbox{if} \ x \in \mathrm{H}_i(A_1,k) \ \mbox{and} \ y \in \mathrm{H}_j(A_2,k)  \\
                                  0 & \mbox{otherwise}.
                                \end{array}\]
Under the appropriate identifications one can consider $id_{ij}$ as an element of $E^{0,t}_p(\mathrm{H}_t(A,k))$ and denote it by $[id_{ij}]$. We then have $[id^t]=\sum_{i+j=t}[id_{ij}]$. \\
Now, consider the map
\[ P_{i,j}: {}^{1}E^{0,i}_r(\mathrm{H}_i(A_1,k)) \otimes_k {}^{2}E^{0,j}_r(\mathrm{H}_j(A_2,k)) \rightarrow E^{0,t}_r(\mathrm{H}_{t}(A,k)). \]
One can check that $P_{i,j}([{}^{1}id^i]\otimes [{}^{2}id^j])=[id_{ij}]$ for all $i+j=t$, so the sum formula for the identity implies
\[ [id^t] = \sum_{i+j=t} P_{i,j}([{}^{1}id^i]\otimes [{}^{2}id^j]). \]
It then follows from Lemma \ref{lemma: product formula} and the definition of characteristic classes that,
\begin{eqnarray*}
v_p^t & = & d_p^{0,t}\Big(  \sum_{i+j=t} P_{i,j}([{}^{1}id^i]\otimes [{}^{2}id^j])\Big) \\
& = &  \sum_{i+j=t} d_p^{0,t}\Big(   P_{i,j}([{}^{1}id^i]\otimes [{}^{2}id^j])\Big) \\
& = & \sum_{i+j=t} \Big( P_{i,j}({}^{1}d_p^{0,i}([{}^{1}id^i])\otimes [{}^{2}id^j])+ (-1)^i P_{i,j}([{}^{1}id^i]\otimes {}^{2}d_p^{0,j}([{}^{2}id^j)])\Big) \\
& = & \sum_{i+j=t} \Big( P_{i,j}({}^{1}v_p^i\otimes [{}^{2}id^j])+ (-1)^i P_{i,j}([{}^{1}id^i]\otimes {}^{2}v_p^j)\Big).
\end{eqnarray*}
Since ${}^{1}v_p^i$ and ${}^{2}v_p^j$ are zero for all $i,j\leq t$, we have $v_p^t=0$.
It now follows from Theorem \ref{th: sah} that $d_p^{s,t}=0$ for all $s$ and all $B$-modules $M$ with $M^A=M$. We conclude that (\ref{eq: decompext}) is $(t,p+1)$-trivial and this finishes the induction. Thus, the extension
%\[  A \rightarrow A \rtimes C \rightarrow C  \]
is $(t,r)$-trivial, and
\[ v_r^t=\sum_{i+j=m} \Big( P_{i,j}({}^{1}v_r^i\otimes [{}^{2}id^j])+ (-1)^i P_{i,j}([{}^{1}id^i]\otimes {}^{2}v_r^j)\Big).\]
\end{proof}
The following corollary is immediate.
\begin{corollary} Suppose the spectral sequences associated to $A_1 \rtimes  C$ and  $A_2 \rtimes  C$ collapse at the second page, in coefficients $\mathrm{H}_t(A_1,k)$ and $\mathrm{H}_t(A_2,k)$ respectively, for each $t\geq 0$. Then the spectral sequence associated to $A \rtimes C$ will collapse at the second page, for all coefficients $M$ for which $M^A=M$.
\end{corollary}
\section{The Lie algebra case} \label{sec: lie algebra}
It is a well-known fact that the universal enveloping algebra $U(\mathfrak{g})$ of a Lie algebra $\mathfrak{g}$ over a field $k$ is a Hopf algebra and that its (co)homology coincides with the ordinary (Chevalley) (co)homology of $\mathfrak{g}$. One can also verify that the universal enveloping algebra functor $U$ maps a (split) extension of Lie algebras to a (split) short exact sequence of Hopf algebras.

Now, let us consider a split Lie algebra extension

\begin{equation}\label{eq: splitextension}\xymatrix {&0 \ar[r] &\mathfrak{n}\ar[r] &\mathfrak{g} \ar[r]^{\pi} &\mathfrak{h}
\ar@/_1.5pc/[l]\ar[r] &0}.
\end{equation}
There is a Lie algebra homomorphism
$\varphi: \mathfrak{h} \rightarrow \mathrm{Der}(\mathfrak{n})$,
such that $\mathfrak{g}$ is isomorphic to the semi-direct product
$\mathfrak{g}\cong \mathfrak{n} \rtimes_{\varphi} \mathfrak{h}$.
Also, this extension has the associated Hochschild-Serre spectral sequence
\[ E^{p,q}_2(M)= \mathrm{H}^p(\mathfrak{h},\mathrm{H}^q(\mathfrak{n},M)) \Rightarrow \mathrm{H}^{p+q}(\mathfrak{g},M) \]
for all $M \in \mathfrak{g}$-mod.

Throughout this section, we will identify  $\mathfrak{h}\mbox{-mod}$ with $\{ M \in \mathfrak{g}\mbox{-mod} \ | \ M^{\mathfrak{n}}=M \}.$
\begin{definition}\rm Let $t\geq0$ and $r\geq2$. A split Lie algebra extensions is called \emph{$(t,r)$-trivial} if the associated split short exact sequence of universal enveloping algebras is $(t,r)$-trivial.
\end{definition}
We can now reformulate the results of the two previous sections in terms of split Lie algebra extensions.
\begin{theorem}\label{th: sah lie} Let $t\geq0$ and $r\geq2$. Suppose  (\ref{eq: splitextension}) is $(t,r)$-trivial, then the following holds.
\begin{itemize}
\item[(a)] For all $s\geq0$ and for all $M \in \mathfrak{h}\mbox{-mod}$, there is a canonical surjective homomorphism
\[ \theta: E_r^{s,0}(\mathrm{H}^t(\mathfrak{n},M)) \rightarrow E_r^{s,t}(M).\]
\item[(b)] The characteristic class $v_r^t \in E_r^{r,t-r+1}(\mathrm{H}_t(\mathfrak{n},k))$  has the property
\[ d_r^{s,t}(x)= (-1)^s y \cdot v_r^t    \]
$\forall s\geq0, \forall M \in \mathfrak{h}\mbox{-mod} , \forall x \in E_r^{s,t}(M)   \mbox{and} \ \forall y \in  E_r^{s,0}(\mathrm{H}^t(\mathfrak{n},M)) \ \mbox{with} \ \theta(y)=x $.
\item[(c)] $v_r^t$ is completely determined by the previous property.
\item[]
\item[(d)] Let $\sigma: \mathfrak{d} \rightarrow \mathfrak{h}$ be a homomorphism of Lie algebras and set $\varphi'= \varphi\circ \sigma$. Suppose the split extension associated to $\mathfrak{n} \rtimes_{\varphi'} \mathfrak{d} $ is $(t,r)$-trivial and denote its characteristic classes by $w_r^t$. Then $v_r^t$ maps to $w_r^t$ under the map induced by $\sigma$ on the spectral sequences. In particular,  $v_2^t$ maps to $w_2^t$.
\end{itemize}
\end{theorem}
\begin{corollary}
The Hochschild-Serre spectral sequence of (\ref{eq: splitextension}) collapses at the second page for all coefficients $M \in \mathfrak{h}\mbox{-mod}$ if and only if the edge differentials
\[ d_m^{0,t}: E_m^{0,t}(\mathrm{H}_t(\mathfrak{n},k)) \rightarrow E_m^{m,t-m+1}(\mathrm{H}_t(\mathfrak{n},k)) \]
are zero for all $t\geq 0$  and $m \geq 2$.
In particular,  the spectral sequence will collapse at the second page for all coefficients $M \in \mathfrak{h}\mbox{-mod}$ if and only if it collapses with coefficients $\mathrm{H}_t(\mathfrak{n},k)$, for all $t\geq 0$.
\end{corollary}
Next, let us suppose that $\mathfrak{n}$ is a direct sum of two Lie algebras
$\mathfrak{n} = \mathfrak{n}_1 \oplus \mathfrak{n}_2$,
and that we have two Lie algebra homomorphisms
$\varphi_1 : \mathfrak{h} \rightarrow \mathrm{Der}(\mathfrak{n}_1)$ and
$\varphi_2 : \mathfrak{h} \rightarrow \mathrm{Der}(\mathfrak{n}_2)$.
Because $\mathrm{Der}(\mathfrak{n}_1) \oplus \mathrm{Der}(\mathfrak{n}_2)$ embeds in $\mathrm{Der}(\mathfrak{n})$, we can use $\varphi_1$ and $\varphi_2$ to obtain a Lie algebra homomorphism
$\varphi: \mathfrak{h} \rightarrow \mathrm{Der}(\mathfrak{n})$.
We then have three split Lie algebra extensions; namely, $0 \rightarrow \mathfrak{n}_i \rightarrow \mathfrak{n}_i\rtimes_{\varphi_i} \mathfrak{h} \rightarrow \mathfrak{h} \rightarrow 0$
for $i=1,$ $2$, and
$0 \rightarrow \mathfrak{n} \rightarrow \mathfrak{n}\rtimes_{\varphi} \mathfrak{h} \rightarrow \mathfrak{h} \rightarrow 0.$
By translating this situation to the universal enveloping algebras, we find ourselves in the set-up of section \ref{sec: decomp}. We are ready to reformulate Theorem \ref{th: decomp} for Lie algebra extensions.
\begin{theorem}\label{th: decomplie} Let $t\geq 0$ and $r\geq 2$.
Suppose the characteristic classes ${}^{1}v_p^i$ and ${}^{2}v_p^j$, of ${}^{1}E_{\ast}(\mathrm{H}_i(\mathfrak{n}_1,k))$ and  ${}^{2}E_{\ast}(\mathrm{H}_j(\mathfrak{n}_2,k))$ respectively, are zero for all $i,j\leq t$ and $2\leq p\leq r-1$. Then the split extension
$0 \rightarrow \mathfrak{n} \rightarrow \mathfrak{g} \rightarrow \mathfrak{h} \rightarrow 0$
is $(t,r)$-trivial. Furthermore, we have a decomposition formula
\[ v_r^t=\sum_{i+j=m} \Big( P_{i,j}({}^{1}v_r^i\otimes [{}^{2}id^j])+ (-1)^i P_{i,j}([{}^{1}id^i]\otimes {}^{2}v_r^j)\Big). \]
\end{theorem}
We will now discuss some corollaries of theorems \ref{th: sah lie} and \ref{th: decomplie}.
\begin{corollary} \label{cor: factor}Suppose the Hochschild-Serre spectral sequences of $\mathfrak{n}_1 \rtimes_{\varphi_1} \mathfrak{h}$ and  $\mathfrak{n}_2 \rtimes_{\varphi_2} \mathfrak{h}$ collapse at the second page, in coefficients $\mathrm{H}_t(\mathfrak{n}_1,k)$ and $\mathrm{H}_t(\mathfrak{n}_2,k)$ respectively, for each $t\geq 0$. Then the Hochschild-Serre spectral sequence of $\mathfrak{n} \rtimes_{\varphi} \mathfrak{h}$ will collapse at the second page, for all coefficients $M$ in $\mathfrak{h}\mbox{-mod}$.
\end{corollary}
\begin{proof} This is immediate from Theorem  \ref{th: decomplie}.
\end{proof}
\begin{corollary}\label{cor: der} Let $k$ be a field of characteristic zero and suppose $ 0 \rightarrow \mathfrak{n} \rightarrow \mathfrak{n} \rtimes_{\varphi} \mathfrak{h} \rightarrow \mathfrak{h}$ is a split extension of finite dimensional Lie algebras. Assume that $\mathrm{Der}(\mathfrak{n})$ has a semi-simple Lie subalgebra $\mathfrak{s}$ such that $\varphi$ factors through $\mathfrak{s}$, i.e.
$\varphi: \mathfrak{h} \rightarrow \mathfrak{s} \subseteq \mathrm{Der}(\mathfrak{n})$.
Then, the Hochschild-Serre spectral sequence $(E_{\ast}(M),d_{\ast})$ associated to the split extension will collapse at the second page for any $\mathfrak{h}$-module $M$.
\end{corollary}

\begin{proof}
The Lie algebra homomorphism $\varphi: \mathfrak{h} \rightarrow \mathfrak{s}$ induces a commutative diagram
\[ \xymatrix{ 0 \ar[r] & \mathfrak{n} \ar[r]  & \mathfrak{n} \rtimes_{i} \mathfrak{s} \ar[r]  & \mathfrak{s} \ar[r] & 0 \\
0 \ar[r] & \mathfrak{n} \ar[r] \ar[u]^{\mathrm{Id}} & \mathfrak{n} \rtimes_{\varphi} \mathfrak{h} \ar[u] \ar[r]  & \mathfrak{h} \ar[r] \ar[u]^{\varphi}& 0, } \]
where $i: \mathfrak{s} \rightarrow \mathrm{Der}(\mathfrak{n})$ is just the inclusion. It is well-known fact that the Hochschild-Serre spectral sequence for an extension with semi-simple quotient, in finite dimensional coefficients, will collapse at the second page. This implies that all its characteristic classes are defined and equal zero. If we now use Theorem \ref{th: sah lie}(d) iteratively, starting from $r=2$, we find that all the characteristic classes of the extension in the bottom row are defined and equal to zero for all $t\geq 0, r\geq 2$.
It then follows from Theorem \ref{th: sah lie}(b) that $d_r=0$ for all $r\geq 2$ and all $M \in \mathfrak{h}\mbox{-mod}$.
\end{proof}
\begin{remark} \rm Let $p\geq 3$ and  $\mathfrak{s}$ be an arbitrary finite dimensional semi-simple Lie algebra over a field of characteristic zero. It can be shown (see \cite{Benoist}, \cite{Bajo}) that there exists a $p$-step nilpotent Lie algebra $\mathfrak{n}_1$ and a $p$-step solvable non-nilpotent Lie algebra $\mathfrak{n}_2$ such that $\mathfrak{s}$ is isomorphic to the Levi factor of $\mathrm{Der}(\mathfrak{n}_i)$.
\end{remark}
\begin{corollary} If $\varphi(\mathfrak{h})$ is one-dimensional, then the Hochschild-Serre spectral sequence will collapse at $E_2$ for all coefficients in $\mathfrak{h}$-mod.
\end{corollary}
\begin{proof} This follows from Theorem \ref{th: sah lie}(d), by an analogous argument as in the proof of Corollary \ref{cor: der} and by observing that extensions with one-dimensional quotients always collapse at the second page.
\end{proof}
\begin{definition} \rm Suppose $k$ is a field of characteristic zero. A finite dimensional Lie algebra $\mathfrak{n}$ is called \emph{reductive} if $\mathfrak{n}$ decomposes as a direct sum of simple $\mathfrak{n}$-modules via the adjoint representation.
\end{definition}
It is a standard fact that when $\mathfrak{n}$ is reductive, there is an isomorphism
$ \mathfrak{n}\cong \mathfrak{a}\oplus \mathfrak{s}$,
where $\mathfrak{a}$ is the center of $\mathfrak{n}$ and $\mathfrak{s}$ is semi-simple. Also, one has that $\mathrm{Der}(\mathfrak{n})\cong \mathrm{Der}(\mathfrak{a})\oplus \mathrm{Der}(\mathfrak{s}) $.
\begin{theorem} \label{cor: reduc}
Suppose $k$ is a field of characteristic zero. Consider the split extension
$0 \rightarrow \mathfrak{n} \rightarrow \mathfrak{g} \rightarrow \mathfrak{h} \rightarrow 0,$
and suppose that $\mathfrak{n}$ is a reductive Lie algebra. Then the Hochschild-Serre spectral sequence associated to this extension will collapse at the second page for all $M \in \mathfrak{h}\mbox{-mod}$.
\end{theorem}
\begin{proof} Let $ \mathfrak{n}\cong \mathfrak{a}\oplus \mathfrak{s}$,  where $\mathfrak{a}$ is the center and $\mathfrak{s}$ is semi-simple. By a result of Barnes (see \cite{Barnes1}), it follows that the spectral sequence associated to a finite dimensional split Lie algebra extension with abelian kernel always collapses at the second page. For the semi-simple Lie algebra $\mathfrak{s}$, one has $\mathfrak{s}\cong \mathrm{Der}(\mathfrak{s})$. So, the collapse of
the spectral sequence of the extension $0 \rightarrow \mathfrak{s} \rightarrow \mathfrak{s}\rtimes \mathrm{Der}(\mathfrak{s}) \rightarrow \mathrm{Der}(\mathfrak{s}) \rightarrow 0$ follows immediately. Combining these results with Corollary \ref{cor: factor}, shows that the Hochschild-Serre spectral sequence associated to the extension  $0 \rightarrow \mathfrak{n} \rightarrow\mathfrak{n}\rtimes \mathrm{Der}(\mathfrak{n}) \rightarrow \mathrm{Der}(\mathfrak{n}) \rightarrow 0$ collapses at the second page. Applying Theorem \ref{th: sah lie}(d) finishes the proof.
\end{proof}

%new section
\section{The group case} \label{sec: group}
Recall that the integral group ring $\mathbb{Z}[G]$ of a group $G$ is a Hopf algebra and that the group ring functor maps (split) group extensions to (split) short exact sequences of Hopf algebras. Note that all the necessary definitions and results from sections \ref{sec: pre} and \ref{sec: cohom} remain valid if we work over a principal ideal domain. Now,
consider the split group extension

\begin{equation}\label{eq: splitgroupextension}\xymatrix {&0 \ar[r] &N\ar[r] & G \ar[r]^{\pi} & H
\ar@/_1.5pc/[l]\ar[r] &0}.
\end{equation}
The extension induces a group homomorphism
$\varphi: H \rightarrow \mathrm{Aut}(N)$,
such that $G$ is isomorphic to the semi-direct product
$G\cong N \rtimes_{\varphi} H$.
%Let us recall the definition of group (co)homology.
%\begin{definition}\rm Let $G$ be a group with group ring $\mathbb{Z}[G]$.
%If $M$ is an $G$-module, we define the $n$-th homology of $G$ with coefficients in $M$ as
%\[ \mathrm{H}_n(G,M):= \mathrm{Tor}^{\mathbb{Z}[G]}_n(\mathbb{Z},M)=\mathrm{H}_n(\mathbb{Z}[G],M). \]
%The $n$-th cohomology of $G$ with coefficients in $M$ is defined as
%\[ \mathrm{H}^n(G,M):= \mathrm{Ext}_{\mathbb{Z}[G]}^n(\mathbb{Z},M)=\mathrm{H}^n(G,M). \]
%\end{definition}
Also, to this group extension we associate the {Lyndon-Hochschild-Serre spectral sequence} for every $M \in G\mbox{-mod}$.
Just as in the case of Lie algebras, we will identify $H\mbox{-mod}$ with $\{ M \in G\mbox{-mod} \ | \ M^{N}=M \}$, when considering extensions as the one above.
Next, we explain what we mean by a $(t,r)$-trivial split group extension.
\begin{definition}\rm
Let $t\geq 0, r\geq 2$. We call a split group extension \emph{$(t,r)$-trivial} if the associated split short exact sequence of group rings is $(t,r)$-trivial and if
$\mathrm{H}_{t-1}(N,\mathbb{Z})$ is $\mathbb{Z}$-free.
\end{definition}
Now, let $M$ be an $H$-module and assume that (\ref{eq: splitgroupextension}) is a $(t,r)$-trivial extension for given $t\geq 0$ and $r\geq 2$. Since $\mathrm{H}_{t-1}(N,\mathbb{Z})$ is $\mathbb{Z}$-free, the Universal Coefficient Theorem entails an isomorphism of $H$-modules $\mathrm{H}^t(N,M) \cong \mathrm{Hom}_{\mathbb{Z}}(\mathrm{H}_t(N,\mathbb{Z}),M)$. This isomorphism gives us a non-degenerate $H$-pairing
\[ \mathrm{H}^t(N,M)\otimes_\mathbb{Z} \mathrm{H}_t(N,\mathbb{Z}) \rightarrow M. \]
Hence, Theorem \ref{th: sah} remains valid in this setting.
\begin{theorem}\label{th: sahgroup} Let $t\geq0$ and $r\geq2$ and suppose (\ref{eq: splitgroupextension}) is $(t,r)$-trivial split extension, then the following holds.
\begin{itemize}
\item[(a)] For all $s\geq 0$ and for all $M \in H\mbox{-mod}$, there is a canonical surjective homomorphism
\[ \theta: E_r^{s,0}(\mathrm{H}^t(N,M)) \rightarrow E_r^{s,t}(M).\]
\item[(b)] The characteristic class $v_r^t \in E_r^{r,t-r+1}(\mathrm{H}_t(N,\mathbb{Z}))$  has the property
\[ d_r^{s,t}(x)= (-1)^s y \cdot v_r^t    \]
$\forall s\geq0, \forall M \in H\mbox{-mod}, \forall x \in E_r^{s,t}(M)   \mbox{and} \ \forall y \in  E_r^{s,0}(\mathrm{H}^t(N,M)) \ \mbox{with} \ \theta(y)=x $.
\item[(c)] $v_r^t$ is completely determined by the previous property.
\item[]
\item[(d)]  Let $\sigma: P \rightarrow H$ be a group homomorphism and set $\varphi'= \varphi\circ \sigma$. Assume the split extension associated to $N \rtimes_{\varphi'} P $ is $(t,r)$-trivial and denote its characteristic classes by $w_r^t$. Then $v_r^t$ maps to $w_r^t$ under the map induced by $\sigma$ on the spectral sequences. In particular,  $v_2^t$ maps to $w_2^t$.
\end{itemize}
\end{theorem}
This is essentially the statement of Theorem 3 of Sah in \cite{Sah}. The only difference is that we expand to split extensions where the kernel is no longer required to be an integral lattice, but just to have a $\mathbb{Z}$-free integral homology group in a specific dimension.

\begin{corollary} Let $N$ be a group with $\mathbb{Z}$-free integral homology.
Then the Lyndon-Hochschild-Serre spectral sequence of the split extension (\ref{eq: splitgroupextension})
collapses at the second page for all coefficients $M \in H\mbox{-mod}$ if and only if the edge differentials
\[ d_m^{0,t}: E_m^{0,t}(\mathrm{H}_t(N,\mathbb{Z})) \rightarrow E_m^{m,t-m+1}(\mathrm{H}_t(N,\mathbb{Z})) \]
are zero for all $t\geq 0$  and $m \geq 2$.
In particular,  the spectral sequence will collapse at the second page for all coefficients $M \in H\mbox{-mod}$ if and only if it collapses with coefficients $\mathrm{H}_t(N,\mathbb{Z})$, for all $t\geq 0$.
\end{corollary}
From now on, let us assume that $N$ is a product of two groups $N_1$ and $N_2$, both with $\mathbb{Z}$-free integral homology,
and that we have two group homomorphisms
$\varphi_1 : H \rightarrow \mathrm{Aut}(N_1)$ and
$\varphi_2 : H \rightarrow \mathrm{Aut}(N_2)$,
which give rise to a group homomorphism
$\varphi: H \rightarrow \mathrm{Aut}(N)$.
We obtain the split group extensions
\[ 0 \rightarrow N_i \rightarrow N_i\rtimes_{\varphi_i} H \rightarrow H \rightarrow 0 \]
for $i=1,2$, and
\[ 0 \rightarrow N \rightarrow N \rtimes_{\varphi} H \rightarrow H \rightarrow 0. \]
By applying the group ring functor to these extensions and by noting that the K\"{u}nneth formula for homology remains valid, we can reformulate the decomposition theorem in terms of split group extensions. This is a generalization of the result obtained in \cite{Petrosyan}.
\begin{theorem}\label{th: decompgroup} Let $t\geq 0$ and $r\geq 2$.
Suppose the characteristic classes ${}^{1}v_p^i$ and ${}^{2}v_p^j$, of ${}^{1}E_{\ast}(\mathrm{H}_i(N_1,\mathbb{Z}))$ and  ${}^{2}E_{\ast}(\mathrm{H}_j(N_2,\mathbb{Z}))$ respectively, are zero for all $i,j\leq t$ and $2\leq p\leq r-1$. Then the split extension
$0 \rightarrow N \rightarrow G \rightarrow H \rightarrow 0$
is $(t,r)$-trivial. Furthermore, we have a decomposition formula
\[ v_r^t=\sum_{i+j=m} \Big( P_{i,j}({}^{1}v_r^i\otimes [{}^{2}id^j])+ (-1)^i P_{i,j}([{}^{1}id^i]\otimes {}^{2}v_r^j)\Big). \]
\end{theorem}
\begin{corollary} \label{cor: factorgroup}Suppose the Lyndon-Hochschild-Serre spectral sequences of $N_1 \rtimes_{\varphi_1} H$ and  $N_2 \rtimes_{\varphi_2} H$ collapse at the second page, in coefficients $\mathrm{H}_t(N_1,\mathbb{Z})$ and $\mathrm{H}_t(N_2,\mathbb{Z})$ respectively, for each $t\geq 0$. Then the Lyndon-Hochschild-Serre spectral sequence of $N \rtimes_{\varphi} H$ will collapse at the second page, for all coefficients $M \in H\mbox{-mod}$.
\end{corollary}
\begin{proof} This is immediate from Theorem  \ref{th: decompgroup}.
\end{proof}
\begin{remark} \rm Note that the preceding results remain valid when $\mathbb{Z}$ is replaced by any principal ideal domain. For example, if we work over a field $F$, the split group extension $0\rightarrow N \rightarrow G \rightarrow H \rightarrow 0$ is \emph{$(t,r)$-trivial} if the associated split short exact sequence of group rings  $F[N] \rightarrow F[G] \rightarrow F[H]$ is $(t,r)$-trivial.
\end{remark}
From now on, we restrict to the extension
\begin{equation}\label{eq: splitgroupextensionlattice}0 \rightarrow L \rightarrow G=L \rtimes_{\varphi}H \rightarrow  H \rightarrow 0,
\end{equation}
where $L$ is an $n$-dimensional integral lattice, and $\varphi: H \rightarrow \mathrm{GL}(n,\mathbb{Z})$
is an integral representation of $H$.

Since $\mathrm{H}_t(L,\mathbb{Z})\cong\Lambda^t(L)$ for all $t$, $L$ has $\mathbb{Z}$-free homology. Hence, as before, we can define the characteristic classes. It turns out that all characteristic classes will have finite order.
\begin{lemma}{\normalfont(Liebermann,\cite{Sah})} Let $t\geq r\geq2$. Consider the extension (\ref{eq: splitgroupextensionlattice}), let $M$ be a finite dimensional $\mathbb{Z}$-free $H$-module and denote the associated Lyndon-Hochschild-Serre spectral sequence by $(E_{\ast}(M),d_{\ast})$. Then, for each $s\geq 0$, the image of $d_r^{s,t}$ is a torsion group annihilated by the integers $m^{t-r+1}(m^{r-1}-1)$,
for all $m \in \mathbb{Z}$  (when $t<r$, $d_r^{s,t}=0$).
\end{lemma}
Using this fact, Sah proves the following result about the order of the characteristic classes $v_r^t$ associated to (\ref{eq: splitgroupextensionlattice}).
\begin{proposition}\label{lemma: sah}{\normalfont(Sah,\cite{Sah})} Let $t\geq r\geq2$, then the order of $v_r^t$ is a divisor of $B_r^t$, where $B_r^t=2$ when $r$ is even and  $B_r^t=\prod p^{\lambda_r^t(p)}$ when $r$ is odd, where $p$ ranges over primes such that $(p-1)| (r-1)$ and $\lambda_r^t(p)=\min(t-r+1,\varepsilon(p)+\mathrm{ord}_p(\frac{r-1}{p-1})$), where $\varepsilon(p)=1$ for odd $p$ and $\varepsilon(2)=2$.
\end{proposition}
In what follows, we intend to improve this result.
\begin{definition} \rm Let $t\geq r\geq2$. If $p$ is a prime such that $p-1$ divides $r-1$, then we define
 \[ \xi_r(p)= \left\{ \begin{array}{ccc}
                    \lambda_r^t(p) & \mbox{if} &   \mbox{$r$ is odd} \\
                    1 & &  \mbox{otherwise.}
                   \end{array}\right.  \]

Now, suppose that (\ref{eq: splitgroupextensionlattice}) is $(k,r)$-trivial for all $2 \leq k\leq t$. We define the following numbers iteratively from $k=r$ to $k=t$
\begin{eqnarray*}
\chi^r_r &=& \left\{ \begin{array}{ccc}
                    p & \mbox{if} &   v_r^r\neq 0 \ \mbox{and} \ r=p^n \ \mbox{for some prime $p$ with  $(p-1)| (r-1)$} \\
                    1 &  &  \mbox{otherwise}
                   \end{array}\right. \\
\chi^k_r &=& \left\{ \begin{array}{ccc}
                    p & \mbox{if} &   v_r^k\neq 0 \ \mbox{and} \ k=p^n \ \mbox{for some prime $p$ with  $(p-1)| (r-1)$} \\
                    & & \mbox{and $\mathrm{ord}_p(\prod_{i=r}^{k-1} \chi_r^i)< \xi_r(p)$} \\
                    1 &  &  \mbox{otherwise.}
                   \end{array}\right.
\end{eqnarray*}

\end{definition}
We will use the following well-known property of the binomial coefficients, which can be seen as a consequence of Lucas' Lemma.
\begin{lemma} \label{lemma: lucas}Let $k\geq2$ be an integer, then
\[ \gcd\Big( {k \choose i} \ | \ i \in \{ 1, \ldots, k-1 \} \Big) = \left\{ \begin{array}{ccc}
                                                                              p  \ & \mbox{if $k=p^n$ for some prime $p$ and some $n \in \mathbb{N}$} \\
                                                                               1   & \mbox{otherwise.}
                                                                            \end{array}\right.  \]
\end{lemma}
\begin{theorem}\label{th: prime order char} Let $t\geq r\geq2$. Suppose that (\ref{eq: splitgroupextensionlattice}) is $(k,r)$-trivial for all $2 \leq k\leq t$.
Then, the order of $v_r^t$ is a divisor of $\prod_{k=r}^t \chi_r^k$.
\end{theorem}
\begin{proof} Since $\mathrm{H}_i(L,\mathbb{Z})=\Lambda^i(L)$, we have an $H$-pairing
\[\mathrm{H}_i(L,\mathbb{Z})\otimes_{\mathbb{Z}} \mathrm{H}_j(L,\mathbb{Z}) \rightarrow  \mathrm{H}_{i+j}(L,\mathbb{Z}) \]
for all $i,j\geq 0$. We will use the multiplicative structure of the Lyndon-Hochschild-Serre spectral sequence associated to (\ref{eq: splitgroupextensionlattice}) induced by this pairing to prove the theorem by induction on $t\geq r$.

First, suppose that $t=r$. If $v_r^r=0$ then we are done, otherwise take $[id^i] \in \mathrm{H}^0(H,\mathrm{H}^i(L,\mathrm{H}_i(L,\mathbb{Z})))$ and $[id^j] \in \mathrm{H}^0(H,\mathrm{H}^j(L,\mathrm{H}_j(L,\mathbb{Z})))$ for all $i,j\geq 1$ such that $i+j=r$. Then one can check that $[id^i]\cdot[id^j] = {r \choose i}[id^r]$. Applying the differential $d_r^{0,r}$ and using the product rule, we find
\[ {r \choose i}v_r^r= {r \choose i}d_r^{0,r}([id^r])= d_r^{0,i}([id^i])\cdot [id^j] + (-1)^i[id^i]\cdot d_r^{0,j}([id^j]). \]
Since $d_r^{p,q}=0$ for all $q<r$, we see that  ${r \choose i}v_r^r=0$, and this is for all $i \in \{1, \ldots, r-1\}$. Hence, Proposition \ref{lemma: sah} and Lemma \ref{lemma: lucas} imply that $\chi^r_r v_r^r=0$.

Now, assume that $v_r^{s}$ has order dividing  $\prod_{k=r}^{s} \chi_r^k$ for all $s \in \{r,r+1,\ldots,t-1\}$.  If $v_r^t=0$, then we are done, otherwise using again the pairing and applying the product rule, we  obtain
\begin{eqnarray*}  {t \choose i}v_r^t & = & d_r^{0,i}([id^i])\cdot [id^j] + (-1)^i[id^i]\cdot d_r^{0,j}([id^j]) \\
& = & v_r^i \cdot [id^j] + (-1)^i[id^i] \cdot [v_r^j]
\end{eqnarray*}
for all  $i,j\geq 1$ such that $i+j=t$. It now follows that
$\prod_{k=r}^{t-1} \chi_r^k {t \choose i} v_r^t = 0$
for all  $i \in \{1, \ldots, t-1\}$. Since \[\gcd\Big( \prod_{k=r}^{t-1} \chi_r^k{t \choose i} \ | \ i \in \{ 1, \ldots, t-1 \} \Big)= \prod_{k=r}^{t-1} \chi_r^k\gcd\Big( {t \choose i} \ | \ i \in \{ 1, \ldots, t-1 \} \Big),\] Proposition \ref{lemma: sah} and Lemma \ref{lemma: lucas} imply that $v_r^t$ has order dividing $\prod_{k=r}^{t} \chi_r^k$.
\end{proof}
\begin{remark} \rm Note that $\prod_{k=r}^{t} \chi_r^k$ is a divisor of $B_r^t$.
\end{remark}
\begin{corollary}If $v_r^{p^n}=0$ for all primes $p$ with $(p-1)| (r-1)$, for all $n \in \mathbb{N}_0$ and all $r\geq 2$. Then the Lyndon-Hochschild-Serre spectral sequence associated to (\ref{eq: splitgroupextensionlattice}) collapses at the second page for all $H$-modules $M$.
\end{corollary}
\begin{proof} We will use induction on $r$ to show that all differentials $d_r$ for $r\geq 2$ are zero for all coefficients in $H$-mod. Suppose $r=2$. By assumption, we have $\chi_2^t=1$ for all relevant $t$. So, Theorem \ref{th: prime order char} implies that $v_2^t=0$ for all $t$. It then follows from Theorem \ref{th: sahgroup}(b) that $d_2^{s,t}=0$ for all $s,t$ and all coefficients in $H$-mod. Now, assume $d_{k}^{s,t}=0$ for all $k \in \{2,\ldots r-1\}$, all $s,t$ and all coefficients in $H$-mod. In particular, (\ref{eq: splitgroupextensionlattice}) is $(t,r)$-trivial for all $t$. It again follows that $\chi_r^t=1$ for all relevant $t$. Therefore, Theorem \ref{th: prime order char} shows that $v_r^t=0$ for all $t$. Using Theorem \ref{th: sahgroup}(b), we conclude $d_r^{s,t}=0$ for all $s,t\geq 0$ and all coefficients in $H$-mod.
\end{proof}
Next, we show how the indexes of subgroups of $H$ can be useful in determining the order of the characteristic classes of (\ref{eq: splitgroupextensionlattice}).
\begin{corollary} \label{th: subgroups}Let $t\geq r\geq2$ and suppose the extension (\ref{eq: splitgroupextensionlattice}) is $(t,r)$-trivial. If $K$ is a  subgroup of $H$, then the extension restricted to $K$ is also $(t,r)$-trivial. Denoting its characteristic class by $w_r^t$ and assuming $[H:K]<\infty$, we have that $v_r^t$ has order dividing $[H:K]\mathrm{ord}(w_r^t)$.
\end{corollary}
\begin{proof} We have the following commutative diagram
\[ \xymatrix{ 0 \ar[r] & L \ar[r] \ar[d]^{id}  & L \rtimes_{\varphi'} K \ar[r] \ar[d] & K \ar[r]\ar[d]^{i} & 0 \\
0 \ar[r] & L \ar[r]  & L\rtimes_{\varphi} H   \ar[r] & H \ar[r] & 0 ,} \]
where $i$ is the inclusion of $K$ into $H$, and $\varphi'=\varphi \circ i$. The fact that the extension in the top row is also $(t,r)$-trivial follows directly from Theorem \ref{th: sahgroup}(d). It also shows that $i^{\ast}(v_r^t)=w_r^t$, where
\[ i^{\ast}: \mathrm{H}^r(H,\mathrm{H}^{t-r+1}(L,\mathrm{H}_t(L,\mathbb{Z}))) \rightarrow \mathrm{H}^r(K,\mathrm{H}^{t-r+1}(L,\mathrm{H}_t(L,\mathbb{Z}))) \]
is the restriction map induced by $i$. If $K$ is a finite index subgroup of $H$, we also have a transfer map
\[  \mathrm{tr}^{\ast}: \mathrm{H}^r(K,\mathrm{H}^{t-r+1}(L,\mathrm{H}_t(L,\mathbb{Z}))) \rightarrow \mathrm{H}^r(H,\mathrm{H}^{t-r+1}(L,\mathrm{H}_t(L,\mathbb{Z}))), \]
with the property that $\mathrm{tr}^{\ast} \circ i^{\ast} = [H:K]id. $ This gives $[H:K]v_r^t= \mathrm{tr}^{\ast}(w_r^t)$ which implies that $[H:K]\mathrm{ord}(w_r^t)v_r^t=0$.
\end{proof}
%\begin{corollary} Let $t\geq r\geq2$. If $H$ has two subgroups $H_1$ and $H_2$ with coprime indexes such that the extensions restricted to $H_1$ and $H_2$ are $(t,r)$-trivial, then (\ref{eq: splitgroupextensionlattice}) is $(t,r)$-trivial.
%\end{corollary}
%\begin{proof} We will prove this by induction on $r$. If $r=2$ then we are done, since every extension is $(t,2)$-trivial. Now suppose the result is true for some $r \geq 2$.
%Fix a $t \geq r+1$ and assume that the extensions restricted to $H_1$ and $H_2$ are $(t,r+1)$-trivial. Then they are of course also $(t,r)$-trivial, so by hypothesis we know that (\ref{eq: splitgroupextensionlattice}) is $(t,r)$-trivial. We need to show that (\ref{eq: splitgroupextensionlattice}) is $(t,r+1)$-trivial. For this it suffices to show that $v_{r}^t=0$.
%Because the extensions restricted to $H_1$ and $H_2$ are $(t,r+1)$-trivial, their $(t,r)$-characteristic classes are zero, so it  follows from Theorem \ref{th: subgroups} that $v_r^t$ has order dividing $[H:H_1]$ and $[H:H_2]$.  Since $[H:H_1]$ and $[H:H_2]$ are coprime we conclude that $v_r^t=0$.
%\end{proof}

\end{document}